\documentclass[reqno]{amsart}
\usepackage{amsmath, amsthm, amssymb, amstext}

\usepackage{hyperref,xcolor}
\hypersetup{
 pdfborder={0 0 0},
 colorlinks,
}

\usepackage{enumitem}
\newtheorem{theorem}{Theorem}
\newtheorem{remark}[theorem]{Remark}

\newtheorem{proposition}[theorem]{Proposition}
\newtheorem{corollary}[theorem]{Corollary}

\DeclareMathOperator*{\ints}{int}         %
\DeclareMathOperator*{\essinf}{ess ~inf}         %
\DeclareMathOperator*{\ww}{w}         %
\DeclareMathOperator*{\Ss}{S}         %

\newcommand{\N}{\mathbb{N}}

\newcommand{\R}{\mathbb{R}}

\newcommand{\Lp}[1]{L^{#1}(\Omega)}

\newcommand{\Wp}[1]{W^{1,#1}(\Omega)}

\newcommand{\Wpzero}[1]{W^{1,#1}_0(\Omega)}

\newcommand{\lan}{\langle}
\newcommand{\ran}{\rangle}
\newcommand{\eps}{\varepsilon}
\newcommand{\ph}{\varphi}

\newcommand{\into}{\int_{\Omega}}
\newcommand{\weak}{\overset{\ww}{\to}}

\newcommand{\Linf}{L^{\infty}(\Omega)}
\newcommand{\close}{\overline{\Omega}}
\newcommand{\interior}{\ints \left(C^1_0(\overline{\Omega})_+\right)}

\newcommand{\cprime}{$'$}

\renewcommand{\l}{\left}
\renewcommand{\r}{\right}

\numberwithin{theorem}{section}
\numberwithin{equation}{section}

\title[Singular Dirichlet $(p,q)$-equations]{Singular Dirichlet $(p,q)$-equations}

\author[N.\,S.\,Papageorgiou]{Nikolaos S.\,Papageorgiou}
\address[N.\,S.\,Papageorgiou]{National Technical University, Department of Mathematics, Zografou Campus, Athens 15780, Greece}
\email{npapg@math.ntua.gr}

\author[P.\,Winkert]{Patrick Winkert}
\address[P.\,Winkert]{Technische Universit\"{a}t Berlin, Institut f\"{u}r Mathematik, Stra\ss e des 17.\,Juni 136, 10623 Berlin, Germany}
\email{winkert@math.tu-berlin.de}


\begin{document}

\begin{abstract}
    We consider a nonlinear Dirichlet problem driven by the $(p,q)$-Laplacian and with a reaction having the combined effects of a singular term and of a parametric $(p-1)$-superlinear perturbation. We prove a bifurcation-type result describing the changes in the set of positive solutions as the parameter $\lambda>0$ varies. Moreover, we prove the existence of a minimal positive solution $u^*_\lambda$ and study the monotonicity and continuity properties of the map $\lambda \to u^*_\lambda$.
\end{abstract}

\maketitle

\section{Introduction}

In a recent paper, the authors \cite{Papageorgiou-Winkert-2019} studied the following singular parametric $p$-Laplacian Dirichlet problem
\begin{equation*}
    \begin{aligned}
	-\Delta_p u   & =u^{-\eta} +\lambda f(x,u)\quad && \text{in } \Omega,\\
	u & = 0  &&\text{on } \partial \Omega,\\
	u>0, \quad \lambda&>0, \quad 0<\eta<1, \quad 1<p. &&
    \end{aligned}
\end{equation*}
They proved a result describing the dependence of the set of positive solutions as the parameter $\lambda>0$ varies, assuming that $f(x,\cdot)$ is $(p-1)$-superlinear.

In the present paper, we consider a singular parametric Dirichlet problem driven by the $(p,q)$-Laplacian, that is, the sum of a $p$-Laplacian and of a $q$-Laplacian with $1<q<p$. To be more precise, the problem under consideration is the following
\begin{equation}\tag{P$_\lambda$}\label{problem}
    \begin{aligned}
	-\Delta_p u-\Delta_q u   & =u^{-\eta} +\lambda f(x,u)\quad && \text{in } \Omega,\\
	u & = 0  &&\text{on } \partial \Omega,\\
	u>0, \quad \lambda&>0, \quad 0<\eta<1, \quad 1<q<p,&&
    \end{aligned}
\end{equation}
where $\Omega \subseteq \R^N$ is a bounded domain with a $C^2$-boundary $\partial\Omega$.
In this problem, the differential operator is not homogeneous and so many of the techniques used in Papageorgiou-Winkert \cite{Papageorgiou-Winkert-2019} are not applicable here. More precisely, in the proof of Proposition 3.1 in \cite{Papageorgiou-Winkert-2019}, the homogeneity of the $p$-Laplacian is crucial in the argument. It provides naturally an upper solution $\overline{u}$ which is an appropriate multiple of the unique solution $e \in \interior$ of problem (3.2) in \cite{Papageorgiou-Winkert-2019} (see also the argument in (3.7)). In our setting, this is no longer possible since the differential operator, so the $(p,q)$-Laplacian, is not homogeneous. This makes our proof here of the fact that $\mathcal{L} \neq \emptyset$ (existence of admissible parameters, see Proposition 3.1) more involved and requires some preparation which involves Propositions \ref{prop_2} and \ref{prop_3}. Moreover, the proof that the critical parameter $\lambda^*>0$ is finite differs for the same reason and here is more involved and requires the use of a different strong comparison principle. In \cite{Papageorgiou-Winkert-2019} (see Proposition 3.6) this is done easily since we can use the spectrum of $(-\Delta_p,\Wpzero{p})$ and in particular the principal eigenvalue $\hat{\lambda}_1>0$ thanks to the homogeneity of the differential operator (see (3.25) in \cite{Papageorgiou-Winkert-2019}). This reasoning fails in our setting and leads to a different geometry near zero (compare hypothesis H(iv) in \cite{Papageorgiou-Winkert-2019} with hypothesis H(iv) in this paper). Furthermore, we now need to employ a different comparison argument based on a recent strong comparison principle due to Papageorgiou-R\u{a}dulescu-Repov\v{s} \cite{Papageorgiou-Radulescu-Repovs-2020}. In addition, the proof of Proposition 3.7 in \cite{Papageorgiou-Winkert-2019} cannot be extended to our problem (see the part from (3.42) and below). The presence of the $q$-Laplacian leads to difficulties. For this reason, our superlinearity condition (see hypothesis H(iii)) differs from the one used in \cite{Papageorgiou-Winkert-2019}. However, we stress that both go beyond the classical Ambrosetti-Rabinowitz condition.

For the parametric perturbation of the singular term, $\lambda f(\cdot,\cdot)$ with $f\colon\Omega\times\R\to\R$, we assume that $f$ is a Carath\'{e}odory function, that is, $x\mapsto f(x,s)$ is measurable for all $s\in \R$ and $s\mapsto f(x,s)$ is continuous for almost all (a.\,a.) $x\in \Omega$. Moreover we assume that $f(x,\cdot)$ exhibits $(p-1)$-superlinear growth as $s\to +\infty$ but it need not satisfy the usual Ambrosetti-Rabinowitz condition (the AR-condition for short) in such cases. Applying variational tools from critical point theory along with suitable truncation and comparison techniques, we prove a bifurcation-type result as in \cite{Papageorgiou-Winkert-2019}, which describes in a precise way the dependence of the set of positive solutions as the parameter $\lambda>0$ changes.

In this direction we mention the recent works of Papageorgiou-R\u{a}dulescu-Repov\v{s} \cite{Papageorgiou-Radulescu-Repovs-2020} and Papageorgiou-Vetro-Vetro \cite{Papageorgiou-Vetro-Vetro-2019} which also deal with nonlinear singular parametric Dirichlet problems. In theses works the parameter multiplies the singular term. Indeed, in Papageorgiou-R\u{a}dulescu-Repov\v{s} \cite{Papageorgiou-Radulescu-Repovs-2020} the equation is driven by a nonhomogeneous differential operator and in the reaction we have the competing effects of a parametric singular term and of a $(p-1)$-superlinear perturbation. In Papageorgiou-Vetro-Vetro \cite{Papageorgiou-Vetro-Vetro-2019} the equation is driven by the $(p,2)$-Laplacian and in the reaction we have the competing effects of a parametric singular term and of a $(p-1)$-linear, resonant perturbation. The work of Papageorgiou-Vetro-Vetro \cite{Papageorgiou-Vetro-Vetro-2019} was continued by Bai-Motreanu-Zeng \cite{Bai-Motreanu-Zeng-2020} where the authors examine the continuity properties with respect to the parameter of the solution multifunction.

Boundary value problems monitored by a combination of differential operators of different nature (such as $(p,q)$-equations), arise in many mathematical processes. We refer, for example, to the works of Bahrouni-R\u{a}dulescu-Repov\v{s} \cite{Bahrouni-Radulescu-Repovs-2019} (transonic flows), Benci-D'Avenia-Fortunato-Pisani \cite{Benci-DAvenia-Fortunato-Pisani-2000} (quantum physics), Cherfils-Il\cprime yasov \cite{Cherfils-Ilyasov-2005} (reaction diffusion systems) and Zhikov \cite{Zhikov-1986} (elasticity theory). We also mention the survey paper of R\u{a}dulescu \cite{Radulescu-2019} on anistropic $(p,q)$-equations.

\section{Preliminaries and Hypotheses} 

The main spaces which we will be using in the study of problem \eqref{problem} are the Sobolev space $\Wpzero{p}$ and the Banach space $C^1_0(\close)$. By $\|\cdot\|$ we denote the norm of the Sobolev space $\Wpzero{p}$ and because of the Poincar\'{e} inequality, we have
\begin{align*}
    \|u\|=\|\nabla u\|_p \quad\text{for all }u\in\Wpzero{p},
\end{align*}
where $\|\cdot\|_p$ denotes norm in $\Lp{p}$ and also in $L^p(\Omega;\R^N)$. From the context it will be clear which one is used. 

The Banach space
\begin{align*}
   C^1_0(\overline{\Omega})= \left\{u \in C^1(\overline{\Omega})\,:\, u\big|_{\partial \Omega}=0 \right\}
\end{align*}
is an ordered Banach space with positive cone
\begin{align*}
   C^1_0(\overline{\Omega})_+=\left\{u \in C^1_0(\overline{\Omega})\,:\, u(x) \geq 0 \text{ for all } x \in \overline{\Omega}\right\}.
\end{align*}
This cone has a nonempty interior given by
\begin{align*}
   \ints \left(C^1_0(\overline{\Omega})_+\right)=\left\{u \in C^1_0(\overline{\Omega})_+: u(x)>0 \text{ for all } x \in \Omega \text{, } \frac{\partial u}{\partial n}(x)<0 \text{ for all } x \in \partial \Omega \right\},
\end{align*}
where $n(\cdot)$ stands for the outward unit normal on $\partial \Omega$.

For every $r\in (1,\infty)$, let $A_r\colon\Wpzero{r}\to W^{-1,r'}(\Omega)=\Wpzero{r}^*$ with $\frac{1}{r}+\frac{1}{r'}=1$ be the nonlinear map defined by
\begin{align}\label{p-Laplace}
    \langle A_r(u), h\rangle=\into |\nabla u|^{r-2}\nabla u \cdot \nabla h\,dx \quad\text{for all }u,h\in\Wpzero{r}.
\end{align}

From Gasi{\'n}ski-Papageorgiou \cite[Problem 2.192, p.\,279]{Gasinski-Papageorgiou-2016} we have the following properties of $A_r$.

\begin{proposition}\label{prop_1}
    The map $A_r\colon\Wpzero{r}\to W^{-1,r'}(\Omega)$ defined in \eqref{p-Laplace} is bounded, that is, it maps bounded sets to bounded sets, continuous, strictly monotone, hence maximal monotone and it is of type $(\Ss)_+$, that is,
    \begin{align*}
	u_n \weak u \text{ in }\Wpzero{r}\quad\text{and}\quad \limsup_{n\to\infty} \langle A_r(u_n),u_n-u\rangle \leq 0,
    \end{align*}
    imply $u_n\to u$ in $\Wpzero{r}$.
\end{proposition}

For $s \in \R$, we set $s^{\pm}=\max\{\pm s,0\}$ and for $u \in W^{1,p}_0(\Omega)$ we define $u^{\pm}(\cdot)=u(\cdot)^{\pm}$. It is well known that
\begin{align*}
    u^{\pm} \in W^{1,p}_0(\Omega), \quad |u|=u^++u^-, \quad u=u^+-u^-.
\end{align*}

For $u,v\in\Wpzero{p}$ with $u(x)\leq v(x)$ for a.\,a.\,$x\in\Omega$ we define
\begin{align*}
    [u,v]&=\big\{h\in\Wpzero{p}: u(x)\leq h(x)\leq v(x)\text{ for a.\,a.\,}x\in\Omega\big\},\\
    [u)&=\big\{h\in\Wpzero{p}: u(x)\leq h(x)\text{ for a.\,a.\,}x\in\Omega\big\}.
\end{align*}

Given a set $S\subseteq \Wp{p}$ we say that it is ``downward directed'', if for any given $u_1, u_2\in S$ we can find $u \in S$ such that $u\leq u_1$ and $u\leq u_2$.

If $h_1,h_2\colon\Omega\to\R$ are two measurable functions, then we write $h_1\prec h_2$ if and only if for every compact $K\subseteq\Omega$ we have $0<c_K\leq h_2(x)-h_1(x)$ for a.\,a.\,$x\in K$.

If $X$ is a Banach space and $\ph\in C^1(X,\R)$, then we define
\begin{align*}
    K_\ph=\left\{u\in X \, : \, \ph'(u)=0\right\}
\end{align*}
being the critical set of $\ph$. Furthermore, we say that $\ph$ satisfies the Cerami condition (C-condition for short), if every sequence $\{u_n\}_{n \geq 1} \subseteq X$ such that $\{\ph(u_n)\}_{n \geq 1}\subseteq \R$ is bounded and such that $\left(1+\|u_n\|_X\right)\ph'(u_n) \to 0$ in $X^*$ as $n \to \infty$, admits a strongly convergent subsequence.

Our Hypotheses on the perturbation $f\colon\Omega\times\R\to\R$ are the following:
\begin{enumerate}[leftmargin=1.2cm]
    \item[H:]
	$f\colon\Omega \times \R\to \R$ is a Carath\'{e}odory function such that $f(x,0)=0$ for a.\,a.\,$x\in\Omega$ and
	\begin{enumerate}[itemsep=0.2cm, topsep=0.2cm]
	    \item[(i)]
		\begin{align*}
		    f(x,s)\leq a(x) \left(1+s^{r-1}\right)
		\end{align*}
		for a.a.\,$x\in\Omega$, for all $s\geq 0$, with $a\in \Linf$ and $p<r<p^*$, where $p^*$ denotes the critical Sobolev exponent with respect to $p$ given by
		\begin{align*}
		    p^*=
		    \begin{cases}
			\frac{Np}{N-p} & \text{if }p<N,\\
			+\infty & \text{if } N \leq p;
		    \end{cases}
		\end{align*}
	    \item[(ii)]
		if $F(x,s)=\int^s_0f(x,t)dt$, then
		\begin{align*}
		    \lim_{s\to +\infty} \frac{F(x,s)}{s^{p}}=+\infty\quad\text{uniformly for a.\,a.\,}x\in\Omega;
		\end{align*}
	    \item[(iii)]
		there exists $\tau \in \left((r-p)\max\left\{\frac{N}{p},1\right\},p^*\right)$ with $\tau>q$ such that
		\begin{align*}
		    0 < c_0\leq \liminf_{s\to +\infty} \frac{f(x,s)s-pF(x,s)}{s^\tau} \quad\text{uniformly for a.\,a.\,}x\in\Omega;
		\end{align*}
	    \item[(iv)]
		\begin{align*}
		    \lim_{s\to 0^+} \frac{f(x,s)}{s^{q-1}}=0\quad\text{uniformly for a.\,a.\,}x\in\Omega
		\end{align*}
		and there exists $\tau \in (q,p)$ such that
		\begin{align*}
			\liminf_{s\to 0^+}\, \frac{f(x,s)}{s^{\tau-1}}\geq \hat{\eta}>0\quad\text{uniformly for a.\,a.\,}x\in\Omega;
		\end{align*}
	    \item[(v)]
		for every $\hat{s}>0$ we have
		\begin{align*}
		    f(x,s) \geq m_{\hat{s}}>0
		\end{align*}
		for a.a.\,$x\in\Omega$ and for all $s\geq \hat{s}$ and for every $\rho>0$ there exists $\hat{\xi}_\rho>0$ such that the function
		\begin{align*}
		    s\to f(x,s)+\hat{\xi}_\rho s^{p-1}
		\end{align*}
		is nondecreasing on $[0,\rho]$ for a.a.\,$x\in\Omega$.
	\end{enumerate}
\end{enumerate}

\begin{remark}
    Since we are looking for positive solutions and the hypotheses above concern the positive semiaxis $\R_+=[0,+\infty)$, without any loss generality, we may assume that
    \begin{align}\label{1}
	f(x,s)=0\quad\text{for a.a.\,}x\in \Omega\text{ and for all }s\leq 0.
    \end{align}
    Hypotheses H(ii), H(iii) imply that
    \begin{align*}
	\lim_{s\to+\infty} \frac{f(x,s)}{s^{p-1}}=+\infty\quad\text{uniformly for a.a.\,}x\in\Omega.
    \end{align*}
    Hence, the perturbation $f(x,\cdot)$ is $(p-1)$-superlinear. In the literature, superlinear equations are usually treated by using the AR-condition. In our case, taking \eqref{1} into account, we refer to a unilateral version of this condition which says that there exist $M>0$ and $\mu>p$ such that
    \begin{align}
	0&<\mu F(x,s) \leq f(x,s)s\quad\text{for a.\,a.\,}x\in \Omega\text{ and for all }s\geq M,\label{2a}\\
	0&<\essinf_\Omega F(\cdot,M).\label{2b}
    \end{align}
    If we integrate \eqref{2a} and use \eqref{2b}, we obtain the weaker condition
    \begin{align*}
	c_1 s^\mu \leq F(x,s)\quad\text{for a.\,a.\,}x\in\Omega,\text{ for all }s\geq M \text{ and for some }c_1>0.
    \end{align*}
    This implies, due to \eqref{2a}, that
    \begin{align*}
	c_1 s^{\mu-1} \leq f(x,s)\quad\text{for a.\,a.\,}x\in\Omega\text{ and for all }s\geq M.
    \end{align*}
    
    We see that the AR-condition is dictating that $f(x,\cdot)$ eventually has $(\mu-1)$-polynomial growth. Here, instead of the AR-condition, see \eqref{2a}, \eqref{2b}, we employ a less restrictive behavior near $+\infty$, see hypothesis H(iii). This way we are able to incorporate in our framework superlinear nonlinearities with ``slower'' growth near $+\infty$. For example, consider the function 
    $f\colon\R\to\R$ (for the sake of simplicity we drop the $x$-dependence) defined by
    \begin{align*}
	f(x)=
	\begin{cases}
	    s^{\mu-1} &\text{if }0 \leq s \leq 1,\\
	    s^{p-1}\ln(x)+s^{\tilde{s}-1} &\text{if } 1<s
	\end{cases}
    \end{align*}
    with $q<\mu<p$ and $\tilde{s}<p$, see \eqref{1}. This function satisfies hypotheses H, but fails to satisfy the AR-condition.
\end{remark}

By a solution of \eqref{problem} we mean a function $u\in\Wpzero{p}$, $u\geq 0$, $u\neq 0$, such that $uh\in \Lp{1}$ for all $h\in \Wpzero{p}$ and
\begin{align*}
    \l\lan A_p(u),h\r\ran +\l\lan A_q(u),h\r\ran =
    \into u^{-\eta}h\,dx+\lambda \into f(x,u)h\,dx\quad \text{for all } h\in\Wpzero{p}.
\end{align*}
The energy functional $\ph_\lambda\colon\Wpzero{p}\to\R$ of the problem \eqref{problem} is given by
\begin{align*}
    \ph_\lambda(u)=\frac{1}{p}\|\nabla u\|_p^p+\frac{1}{q}\|\nabla u\|_q^q -\frac{1}{1-\eta}\into \left(u^+\right)^{1-\eta}\,dx -\lambda \into F\left(x,u^+\right)\,dx
\end{align*}
for all $h\in\Wpzero{p}$.

We can find solutions of \eqref{problem} among the critical points of $\ph_\lambda$. The problem that we face is that because of the third term, so the singular one, the energy functional $\ph_\lambda$ is not $C^1$. So, we cannot apply directly the minimax theorems of the critical point theory on $\ph_\lambda$. Solving related auxiliary Dirichlet problems and then using suitable truncation and comparison techniques, we are able to overcome this difficulty, isolate the singularity and deal with $C^1$-functionals on which the classical critical point theory can be used.

To this end, first we consider the following purely singular Dirichlet problem
\begin{equation}\label{3}
    \begin{aligned}
	-\Delta_p u-\Delta_q u   & =u^{-\eta}\quad && \text{in } \Omega,\\
	u & = 0  &&\text{on } \partial \Omega,\\
	u>0,\quad 0&<\eta<1, && 1<q<p.
    \end{aligned}
\end{equation}
From Proposition 10 of Papageorgiou-R\u{a}dulescu-Repov\v{s} \cite{Papageorgiou-Radulescu-Repovs-2020} we have the following result concerning problem \eqref{3}.

\begin{proposition}\label{prop_2}
    Problem \eqref{3} admits a unique solution $\underline{u}\in\interior$.
\end{proposition}

From the Lemma in Lazer-McKenna \cite{Lazer-McKenna-1991} we know that
\begin{align*}
	\underline{u}^{-\eta} \in \Lp{1}.
\end{align*}
Moreover, from Hardy's inequality we have
\begin{align*}
	\underline{u}^{-\eta} h \in \Lp{1}\quad\text{and}\quad\into \left|\underline{u}^{-\eta}h\right|\,dx \leq \hat{c} \|h\|
\end{align*}
for all $h \in \Wpzero{p}$. It follows that $\underline{u}^{-\eta}+1 \in W^{-1,p'}(\Omega)=\Wpzero{p}^*$.

So, we can consider a second auxiliary Dirichlet problem
\begin{equation}\label{5}
    \begin{aligned}
	-\Delta_p u-\Delta_q u   & =\underline{u}^{-\eta}+1\quad && \text{in } \Omega,\\
	u & = 0  &&\text{on } \partial \Omega,\\
	0&<\eta<1,\quad  1<q<p.
    \end{aligned}
\end{equation}

We show that \eqref{5} has a unique solution.

\begin{proposition}\label{prop_3}
    Problem \eqref{5} admits a unique solution $\overline{u}\in\interior$.
\end{proposition}

\begin{proof}
    Consider the operator $L\colon\Wpzero{p}\to W^{-1,p'}(\Omega)$ with $\frac{1}{p}+\frac{1}{p'}=1$ defined by
    \begin{align*}
		L(u)=A_p(u)+A_q(u) \quad\text{for all }u\in\Wpzero{p}.
    \end{align*}
    This operator is continuous, strictly monotone, hence maximal monotone and coercive. Since $\underline{u}^{-\eta}+1\in W^{-1,p'(\Omega)}$ (see the comments after Proposition \ref{prop_2}), we can find $\overline{u} \in \Wpzero{p}, \overline{u}\neq 0$ such that
    \begin{align*}
		L\left(\overline{u}\right)=\underline{u}^{-\eta}+1.
    \end{align*}
    The strict monotonicity of $L$ implies the uniqueness of $\overline{u}$ while Theorem B.1 of Giacomoni-Schindler-Tak\'{a}\v{c} \cite{Giacomoni-Schindler-Takac-2007} implies that $\overline{u} \in C^1_0(\close)_+\setminus\{0\}$. Furthermore, we have 
    \begin{align*}
		\Delta_p \overline{u}(x)+\Delta_q \overline{u}(x) \leq 0\quad\text{for a.\,a.\,}x\in\Omega.
    \end{align*}
	Hence, from the nonlinear maximum principle, see Pucci-Serrin \cite[pp.\,111 and 120]{Pucci-Serrin-2007}, we conclude that $\overline{u}\in \interior$.
\end{proof}

\section{Positive solutions}

We introduce the following two sets
\begin{align*}
    \mathcal{L}&=\left\{\lambda>0: \text{problem \eqref{problem} has a positive solution}\right\},\\
    \mathcal{S}_\lambda&=\left\{u: u\text{ is a positive solution of problem \eqref{problem}}\right\}.
\end{align*}

\begin{proposition}
    If hypotheses H hold, then $\mathcal{L}\neq \emptyset$.
\end{proposition}

\begin{proof}
    Let $\overline{u}\in\interior$ be as in Proposition \ref{prop_3}. Hypothesis H(i) implies that $f(\cdot,\overline{u}(\cdot))\in\Linf$. So, we can find $\lambda_0>0$ such that
    \begin{align}\label{8}
	0 \leq \lambda_0 f\l(x,\overline{u}(x)\r) \leq 1\quad\text{for a.\,a.\,}x\in\Omega.
    \end{align}
    From the weak comparison principle (see Pucci-Serrin \cite[Theorem 3.4.1, p.\,61]{Pucci-Serrin-2007}), we have $\underline{u} \leq \overline{u}$. So, for given $\lambda \in (0,\lambda_0]$, we can define the following truncation of the reaction of problem \eqref{problem}
    \begin{align}\label{9}
	g_\lambda(x,s)=
	\begin{cases}
	    \underline{u}(x)^{-\eta}+\lambda f(x,\underline{u}(x)) &\text{if }s<\underline{u}(x),\\
	    s^{-\eta}+\lambda f(x,s) &\text{if }\underline{u}(x) \leq s \leq \overline{u}(x),\\
	    \overline{u}(x)^{-\eta}+\lambda f(x,\overline{u}(x)) &\text{if }\overline{u}(x)<s.
	\end{cases}
    \end{align}
    This is a Carath\'{e}odory function. We set $G_\lambda(x,s)=\int_0^s g_\lambda(x,t)\,dt$ and consider the $C^1$-functional $\psi_\lambda\colon\Wpzero{p}\to\R$ defined by
    \begin{align*}
	\psi_\lambda(u)= \frac{1}{p} \|\nabla u\|_p^p+\frac{1}{q}\|\nabla u\|_q^q -\into G_\lambda(x,u)\,dx\quad\text{for all }u \in\Wpzero{p},
    \end{align*}
    see also Papageorgiou-Smyrlis \cite[Proposition 3]{Papageorgiou-Smyrlis-2015}. From \eqref{9} we see that $\psi_\lambda$ is coercive. Also, using the Sobolev embedding theorem, we see that $\psi_\lambda$ is sequentially weakly lower semicontinuous. So, by the Weierstra\ss{}-Tonelli theorem, we can find $u_\lambda\in\Wpzero{p}$ such that
    \begin{align*}
	\psi_\lambda(u_\lambda)=\min \l[\psi_\lambda(u)\,:\,u\in\Wpzero{p}\r].
    \end{align*}
    This means, in particular, that $\psi_\lambda'(u_\lambda)=0$, which gives
    \begin{align}\label{10}
	\l\lan A_p(u_\lambda),h\r\ran+\l\lan A_q(u_\lambda),h\r\ran=\into g_\lambda(x,u_\lambda)h\,dx\quad\text{for all }h\in\Wpzero{p}.
    \end{align}
    First, we choose $h=\l(\underline{u}-u_\lambda\r)^+\in\Wpzero{p}$ in \eqref{10}. This yields, because of \eqref{9}, $f \geq 0$ and Proposition \ref{prop_2} that
    \begin{align*}
	&\l\lan A_p(u_\lambda),\l(\underline{u}-u_\lambda\r)^+\r\ran+\l\lan A_q(u_\lambda),\l(\underline{u}-u_\lambda\r)^+\r\ran\\
	&=\into \l[\underline{u}^{-\eta}+\lambda f(x,\underline{u})\r] \l(\underline{u}-u_\lambda\r)^+\,dx\\
	& \geq \into \underline{u}^{-\eta} \l(\underline{u}-u_\lambda\r)^+\,dx\\
	& =\l\lan A_p(\underline{u}),\l(\underline{u}-u_\lambda\r)^+\r\ran+\l\lan A_q(\underline{u}),\l(\underline{u}-u_\lambda\r)^+\r\ran.
    \end{align*}
    This implies
    \begin{align*}
    &\int_{\{\underline{u}>u_\lambda\}} \l(|\nabla \underline{u}|^{p-2} \nabla \underline{u} - |\nabla u_\lambda|^{p-2}\nabla u_\lambda\r) \cdot \l(\nabla \underline{u}-\nabla u_\lambda\r)\,dx\\
    &+\int_{\{\underline{u}>u_\lambda\}} \l(|\nabla \underline{u}|^{q-2} \nabla \underline{u} - |\nabla u_\lambda|^{q-2}\nabla u_\lambda\r) \cdot \l(\nabla \underline{u}-\nabla u_\lambda\r)\,dx\\
    & \leq 0,
    \end{align*}
    which means $|\{\underline{u}>u_\lambda\}|_N=0$ with $|\cdot|_N$ being the Lebesgue measure of $\R^N$.
    Hence,
    \begin{align}\label{11}
	\underline{u} \leq u_\lambda.
    \end{align}
    Next, we choose $h=\l(u_\lambda-\overline{u}\r)^+\in\Wpzero{p}$ in \eqref{10}. Applying \eqref{9}, \eqref{11}, \eqref{8} and recall that $0 <\lambda\leq\lambda_0$, we obtain
    \begin{align*}
	&\l\lan A_p(u_\lambda),\l(u_\lambda-\overline{u}\r)^+\r\ran+\l\lan A_q(u_\lambda),\l(u_\lambda-\overline{u}\r)^+\r\ran\\
	&=\into \l[\overline{u}^{-\eta}+\lambda f(x,\overline{u})\r] \l(u_\lambda-\overline{u}\r)^+\,dx\\
	& \leq \into \l[ \underline{u}^{-\eta}+1\r]\l(u_\lambda-\overline{u}\r)^+\,dx\\
	&=\l\lan A_p(\overline{u}),\l(u_\lambda-\overline{u}\r)^+\r\ran+\l\lan A_q(\overline{u}),\l(u_\lambda-\overline{u}\r)^+\r\ran\\
    \end{align*}
    From this we see that
    \begin{align*}
    &\int_{\{u_\lambda>\overline{u}\}} \l(|\nabla u_\lambda|^{p-2} \nabla u_\lambda - |\nabla \overline{u}|^{p-2}\nabla \overline{u}\r) \cdot \l(\nabla u_\lambda -\nabla \overline{u}\r)\,dx\\
    &+\int_{\{u_\lambda>\overline{u}\}} \l(|\nabla u_\lambda|^{q-2} \nabla u_\lambda - |\nabla \overline{u}|^{q-2}\nabla \overline{u}\r) \cdot \l(\nabla u_\lambda -\nabla \overline{u}\r)\,dx\\
    & \leq 0
    \end{align*}
	and so $|\{u_\lambda>\overline{u}\}|_N=0$. Thus, $u_\lambda\leq \overline{u}$. So, we have proved that
    \begin{align}\label{12}
	u_\lambda \in [\underline{u},\overline{u}].
    \end{align}
    Then, \eqref{12}, \eqref{9} and \eqref{10} imply that $u_\lambda \in \mathcal{S}_\lambda$ and so $(0,\lambda_0]\subseteq \mathcal{L}\neq \emptyset$.
\end{proof}

\begin{proposition}\label{prop_5}
    If hypotheses H hold and $\lambda\in\mathcal{L}$, then $\underline{u}\leq u$ for all $u \in\mathcal{S}_\lambda$.
\end{proposition}

\begin{proof}
    Let $u \in \mathcal{S}_\lambda$. On $\Omega\times(0,+\infty)$ we introduce the Carath\'{e}odory function $k(\cdot,\cdot)$ defined by
    \begin{align}\label{13}
	k(x,s)=
	\begin{cases}
	    s^{-\eta} &\text{if }0<s\leq u(x),\\
	    u(x)^{-\eta}&\text{if }u(x)<s
	\end{cases}
    \end{align}
    for all $(x,s)\in \Omega\times (0,+\infty)$. Then we consider the following Dirichlet $(p,q)$-problem
    \begin{equation*}
	\begin{aligned}
	    -\Delta_p u -\Delta_q u  & =k(x,u)
	    \quad && \text{in } \Omega,\\
	    u & = 0  &&\text{on } \partial \Omega,\\
	    u&>0, \quad 1<q<p.
	\end{aligned}
    \end{equation*}
    Proposition 10 of Papageorgiou-R\u{a}dulescu-Repov\v{s} \cite{Papageorgiou-Radulescu-Repovs-2020} implies that this problem admits a solution
    \begin{align}\label{14}
	\tilde{\underline{u}}\in \interior.
    \end{align}
    This means
    \begin{align}\label{15}
	\l\lan A_p\l(\tilde{\underline{u}}\r),h\r\ran
	+\l\lan A_q\l(\tilde{\underline{u}}\r),h\r\ran
	=\into k\l(x,\tilde{\underline{u}}\r)h\,dx\quad\text{for all }h\in\Wpzero{p}.
    \end{align}
    Choosing $h=\l(\tilde{\underline{u}}-u\r)^+\in\Wpzero{p}$ in \eqref{15} and applying \eqref{13}, $f \geq 0$ and $u\in\mathcal{S}_\lambda$ gives
    \begin{align*}
	&\l\lan A_p(\tilde{\underline{u}}),\l(\tilde{\underline{u}}-u\r)^+\r\ran
	+\l\lan A_q(\tilde{\underline{u}}),\l(\tilde{\underline{u}}-u\r)^+\r\ran\\
	&=\into u^{-\eta} \l(\tilde{\underline{u}}-u\r)^+\,dx\\
	& \leq \into \l[u^{-\eta}+\lambda f(x,u)\r]\l(\tilde{\underline{u}}-u\r)^+\,dx\\
	& =\l\lan A_p(u),\l(\tilde{\underline{u}}-u\r)^+\r\ran+\l\lan A_q(u),\l(\tilde{\underline{u}}-u\r)^+\r\ran.
    \end{align*}
    This implies
    \begin{align*}
	    &\int_{\{\tilde{\underline{u}}>u\}} \l(|\nabla \tilde{\underline{u}}|^{p-2} \nabla \tilde{\underline{u}} - |\nabla u|^{p-2}\nabla u\r) \cdot \l(\nabla \tilde{\underline{u}}-\nabla u\r)\,dx\\
	    &+\int_{\{\tilde{\underline{u}}>u\}} \l(|\nabla \tilde{\underline{u}}|^{q-2} \nabla \tilde{\underline{u}} - |\nabla u|^{q-2}\nabla u\r) \cdot \l(\nabla \tilde{\underline{u}}-\nabla u\r)\,dx\\
	    & \leq 0,
    \end{align*}
    which means $|\{\tilde{\underline{u}}>u\}|_N=0$. Thus,
    \begin{align}\label{16}
	\tilde{\underline{u}} \leq u.
    \end{align}
    From \eqref{16}, \eqref{14}, \eqref{13}, \eqref{15} and Proposition \ref{prop_2} it follows that $\tilde{\underline{u}}=u$. Therefore, $\underline{u} \leq u$ for all $u \in \mathcal{S}_\lambda$.
\end{proof}

As before, using Theorem B.1 of Giacomoni-Schindler-Tak\'{a}\v{c} \cite{Giacomoni-Schindler-Takac-2007}, we have the following result about the solution set $S_\lambda$.
\begin{proposition}\label{prop_6}
    If hypotheses H hold and $\lambda \in\mathcal{L}$, then $S_\lambda \subseteq \interior$.
\end{proposition}

Let $\lambda^*=\sup \mathcal{L}$.

\begin{proposition}
    If hypotheses H hold, then $\lambda^*<\infty$.
\end{proposition}

\begin{proof}
	Hypotheses H(ii), (iii) imply that we can find $M>0$ such that
	\begin{align*}
		f(x,s) \geq s^{p-1}\quad\text{for a.\,a.\,}x\in \Omega \text{ and for all }s \geq M.
	\end{align*}
	Moreover, hypothesis H(iv) implies that there exist $\delta\in(0,1)$ and $\hat{\eta}_1 \in (0,\hat{\eta})$ such that
	\begin{align*}
		f(x,s) \geq \hat{\eta}_1 s^{\tau-1}\geq \hat{\eta}_1 s^{p-1}
	\end{align*}
	for a.\,a.\,$x\in \Omega$ and for all $0\leq s \leq \delta$ since $\tau<p$ and $ \delta<1$. This yields
	\begin{align*}
		\frac{1}{\hat{\eta}_1} f(x,s) \geq s^{p-1}\quad\text{for a.\,a.\,}x\in \Omega \text{ and for all }0 \leq s \leq \delta.
	\end{align*}
	In addition, on account of hypothesis H(v) we can find $\tilde{\lambda}>0$ large enough such that
	\begin{align*}
		\tilde{\lambda} f(x,s) \geq M^{p-1} \quad\text{for a.\,a.\,}x\in \Omega \text{ and for all } \delta \leq s \leq M.
	\end{align*}
	Therefore, taking into account the calculations above, there exists $\hat{\lambda}>0$ large enough such that
    \begin{align}\label{17}
	s^{p-1}\leq \hat{\lambda}f(x,s)\quad\text{for a.\,a.\,}x\in\Omega \text{ and for all }s \geq 0.
    \end{align}
    Let $\lambda>\hat{\lambda}$ and suppose that $\lambda\in\mathcal{L}$. Then we can find $u_\lambda\in\mathcal{S}_\lambda\subseteq \interior$, see Proposition \ref{prop_6}. Let $\Omega'\subset\subset \Omega$ with $C^2$-boundary $\partial\Omega'$. Then $m_0=\min_{\overline{\Omega'}} u_\lambda>0$ since $u_\lambda \in\interior$. Let $\rho=\|u_\lambda\|_\infty$ and let $\hat{\xi}_\rho>0$ be as postulated by hypothesis H(v). For $\delta>0$, we set $m_0^\delta=m_0+\delta$. Applying \eqref{17}, hypothesis H(v) and $u_\lambda\in\mathcal{S}_\lambda$, we have for a.\,a.\,$x\in\Omega'$
    \begin{align*}
	&-\Delta_p m_0^\delta -\Delta_q m_0^\delta +\lambda \hat{\xi}_\rho \l(m_0^\delta\r)^{p-1}-\lambda \l(m_0^\delta\r)^{-\eta}\\
	&\leq \lambda \hat{\xi}_\rho m_0^{p-1}+\chi(\delta) \quad\text{with }\chi(\delta)\to 0^+ \text{ as }\delta \to 0^+\\
	&\leq \l[\lambda \hat{\xi}_\rho+1\r]m_0^{p-1}+\chi(\delta)\\
	&\leq  \hat{\lambda}f(x,m_0)+\lambda \hat{\xi}_\rho m_0^{p-1}+\chi(\delta)\\
	&= \lambda \l[f(x,m_0)+\hat{\xi}_\rho m_0^{p-1}\r]-\l(\lambda-\hat{\lambda}\r)f(x,m_0) +\chi(\delta)\\
	& \leq \lambda \l[ f\l(x,u_\lambda(x)\r)+\hat{\xi}_\rho u_\lambda(x)^{p-1}\r] \quad\text{for }\delta>0\text{ small enough}\\
	&=-\Delta_p u_\lambda(x)-\Delta_q u_\lambda(x) +\lambda \hat{\xi}_\rho u_\lambda(x)^{p-1}-\lambda u_\lambda(x)^{-\eta}.
    \end{align*}
    Note that for $\delta>0$ small enough, we will have
    \begin{align*}
	0<\hat{\eta} \leq \l[\lambda-\hat{\lambda}\r]f(x,m_0)-\chi(\delta)\quad\text{for a.\,a.\,}x\in\Omega',
    \end{align*}
    see hypothesis H(v). Then, invoking Proposition 6 of Papageorgiou-R\u{a}dulescu-Repov\v{s} \cite{Papageorgiou-Radulescu-Repovs-2020}, it follows that
    \begin{align*}
	m_0^\delta<u_\lambda(x)\quad\text{for a.\,a.\,}x\in\Omega'\text{ and for }\delta>0 \text{ small enough},
    \end{align*}
    which contradicts the definition of $m_0$. Therefore, $\lambda\not\in\mathcal{L}$ and so we conclude that $\lambda^*\leq \hat{\lambda}<\infty$.
\end{proof}

Next, we are going to show that $\mathcal{L}$ is an interval. So, we have
\begin{align*}
    \l(0,\lambda^*\r)\subseteq \mathcal{L}\subseteq \l(0,\lambda^*\r].
\end{align*}

\begin{proposition}
    If hypotheses H hold, $\lambda\in \mathcal{L}$ and $0<\mu<\lambda$, then $\mu\in\mathcal{L}$.
\end{proposition}

\begin{proof}
    Since $\lambda\in\mathcal{L}$, we can find $u_\lambda\in\mathcal{S}_\lambda\subseteq \interior$. We know that $\underline{u}\leq u_\lambda$, see Proposition \ref{prop_5}. So, we can define the following truncation $e_\mu\colon\Omega\times\R\to\R$ of the reaction for problem \eqref{problem}
    \begin{align}\label{18}
	e_\mu(x,s)=
	\begin{cases}
	    \underline{u}(x)^{-\eta}+\mu f(x,\underline{u}(x))&\text{if }s<\underline{u}(x),\\
	    s^{-\eta}+\mu f(x,s)&\text{if }\underline{u}(x)\leq s\leq u_\lambda(x),\\
	    u_\lambda(x)^{-\eta}+\mu f\l(x,u_\lambda(x)\r)&\text{if }u_\lambda(x)<s,
	\end{cases}
    \end{align}
    which is a Carath\'{e}odory function. We set $E_\mu(x,s)=\int^s_0 e_\mu(x,t)\,dt$ and consider the $C^1$-functional $\hat{\ph}_\mu\colon\Wpzero{p}\to\R$ defined by
    \begin{align*}
	\hat{\ph}_\mu(u)=\frac{1}{p}\|\nabla u\|_p^p+\frac{1}{q}\|\nabla u\|_q^q -\into E_\mu(x,u)\,dx\quad\text{for all }u\in\Wpzero{p},
    \end{align*}
    see Papageorgiou-Vetro-Vetro \cite{Papageorgiou-Vetro-Vetro-2019}. From \eqref{18} it is clear that $\hat{\ph}_\mu$ is coercive. Moreover, it is sequentially weakly lower semicontinuous. Therefore, we can find $u_\mu\in\Wpzero{p}$ such that
    \begin{align*}
	\hat{\ph}_\mu\l(u_\mu\r)=
	\min \l[\hat{\ph}_\mu(u)\,:\,u\in\Wpzero{p}\r].
    \end{align*}
    In particular, we have $\hat{\ph}_\mu'\l(u_\mu\r)=0$ which means
    \begin{align}\label{19}
	\l\lan A_p\l(u_\mu\r),h\r\ran+\l\lan A_q\l(u_\mu\r),h\r\ran=\into e_\mu(x,u)h\,dx\quad\text{for all }h\in\Wpzero{p}.
    \end{align}
    Choosing $h=\l(\underline{u}-u_\mu\r)^+\in\Wpzero{p}$ in \eqref{19} and applying \eqref{18}, $f\geq 0$ and Proposition \ref{prop_2} yields
    \begin{align*}
	&\l\lan A_p\l(u_\mu\r),\l(\underline{u}-u_\mu\r)^+\r\ran+\l\lan A_q\l(u_\mu\r),\l(\underline{u}-u_\mu\r)^+\r\ran\\
	&=\into \l[\underline{u}^{-\eta}+\mu f(x,\underline{u})\r]\l(\underline{u}-u_\mu\r)^+\,dx\\
	& \geq \into \underline{u}^{-\eta} \l(\underline{u}-u_\mu\r)^+\,dx\\
	&=\l\lan A_p\l(\underline{u}\r),\l(\underline{u}-u_\mu\r)^+\r\ran+\l\lan A_q\l(\underline{u}\r),\l(\underline{u}-u_\mu\r)^+\r\ran.
    \end{align*}
    We obtain $\underline{u} \leq u_\mu$. Furthermore, choosing $h=\l(u_\mu-u_\lambda\r)^+\in\Wpzero{p}$ in \eqref{19} and applying \eqref{18}, $\mu<\lambda$ and $u_\lambda \in\mathcal{S}_\lambda$, we get
    \begin{align*}
	&\l\lan A_p\l(u_\mu\r),\l(u_\mu-u_\lambda\r)^+\r\ran+\l\lan A_q\l(u_\mu\r),\l(u_\mu-u_\lambda\r)^+\r\ran\\
	&=\into \l[u_\lambda^{-\eta}+\mu f(x,u_\lambda)\r]\l(u_\mu-u_\lambda\r)^+\,dx\\
	& \leq  \into \l[u^{-\eta}+\lambda f(x,u_\lambda)\r]\l(u_\mu-u_\lambda\r)^+\,dx\\
	&=\l\lan A_p\l(u_\lambda\r),\l(u_\mu-u_\lambda\r)^+\r\ran+\l\lan A_q\l(u_\lambda\r),\l(u_\mu-u_\lambda\r)^+\r\ran.
    \end{align*}
    Hence, $u_\mu \leq u_\lambda$ and so we have proved that
    \begin{align}\label{20}
	u_\mu\in \l[\underline{u},u_\lambda\r].
    \end{align}
    From \eqref{20}, \eqref{18} and \eqref{19} we infer that
    \begin{align*}
	u_\mu \in\mathcal{S}_\mu \subseteq \interior.
    \end{align*}
    Thus, $\mu \in\mathcal{L}$.
\end{proof}

A byproduct of the proof above is the following corollary.

\begin{corollary}\label{coro_9}
    If hypotheses H hold, $\lambda\in\mathcal{L}$, $u_\lambda\in\mathcal{S}_\lambda\subseteq \interior$ and $\mu \in (0,\lambda)$, then $\mu\in\mathcal{L}$ and there exists $u_\mu\in\mathcal{S}_\mu\subseteq \interior$ such that $u_\mu \leq u_\lambda$.
\end{corollary}

Using the strong comparison principle of Papageorgiou-R\u{a}dulescu-Repov\v{s} \cite{Papageorgiou-Radulescu-Repovs-2020} we can improve the conclusion of this corollary as follows.

\begin{proposition}\label{prop_10}
    If hypotheses H hold, $\lambda\in\mathcal{L}$, $u_\lambda\in\mathcal{S}_\lambda\subseteq \interior$ and $\mu \in (0,\lambda)$, then $\mu\in\mathcal{L}$ and there exists $u_\mu \in\mathcal{S}_\mu\subseteq \interior$ such that
    \begin{align*}
	u_\lambda-u_\mu\in\interior.
    \end{align*}
\end{proposition}

\begin{proof}
    From Corollary \ref{coro_9} we already have that $\mu\in\mathcal{L}$ and we also know that there exists $u_\mu \in\mathcal{S}_\mu\subseteq\interior$ such that
    \begin{align}\label{21}
	u_\mu \leq u_\lambda.
    \end{align}
    Let $\rho=\|u_\lambda\|_\infty$ and let $\hat{\xi}_\rho>0$ be as postulated by hypothesis H(v). Applying $u_\mu\in\mathcal{S}_\mu$, \eqref{21}, hypothesis H(v) and $\mu<\lambda$, we obtain
    \begin{align}\label{22}
	\begin{split}
	    &-\Delta_p u_\mu(x)-\Delta_q u_\mu(x)+\lambda \hat{\xi}_\rho u_\mu(x)^{p-1}-u_\mu(x)^{-\eta}\\
	    &=\mu f(x,u_\mu(x))+\lambda\hat{\xi}_\rho u_\mu(x)^{p-1}\\
	    & =\lambda \l[f(x,u_\mu(x))+\hat{\xi}_\rho u_\mu(x)^{p-1}\r]-(\lambda-\mu)f(x,u_\mu(x))\\
	    & \leq \lambda \l[f(x,u_\lambda(x)) +\hat{\xi}_\rho u_\lambda(x)^{p-1}\r]\\\
	    & =-\Delta_pu_\lambda(x)-\Delta_q u_\lambda(x) +\lambda \hat{\xi}_\rho u_\lambda(x)^{p-1}-u_\lambda(x)^{-\eta}\quad \text{for a.\,a.\,}x\in\Omega.
	\end{split}
    \end{align}
    Since $u_\mu\in\interior$, because of hypothesis H(v), we have
    \begin{align*}
	0\prec (\lambda-\mu)f(\cdot,u_\mu(\cdot)).
    \end{align*}
    Then, from \eqref{22} and Proposition 7 of Papageorgiou-R\u{a}dulescu-Repov\v{s} \cite{Papageorgiou-Radulescu-Repovs-2020} we conclude that $u_\lambda-u_\mu\in\interior$.
\end{proof}

\begin{proposition}\label{prop_11}
    If hypotheses H hold and $\lambda\in (0,\lambda^*)$, then problem \eqref{problem} has at least two positive solutions
    \begin{align*}
	u_0, \hat{u} \in \interior,\quad u_0\leq \hat{u}, \quad u_0 \neq \hat{u}.
    \end{align*}
\end{proposition}

\begin{proof}
    Let $\lambda <\vartheta <\lambda^*$. Due to Proposition \ref{prop_10}, we can find $u_\vartheta \in \mathcal{S}_\vartheta \subseteq \interior$ and $u_0\in\mathcal{S}_\lambda$ such that
    \begin{align}\label{23}
	u_\vartheta-u_0\in\interior.
    \end{align}
    From Proposition \ref{prop_5} we know that $\underline{u}\leq u_0$. Therefore, $u_0^{-\eta} \in \Lp{1}$. So, we can define the following truncation $w_\lambda\colon\Omega\times\R\to\R$ of the reaction in problem \eqref{problem}
    \begin{align}\label{24}
	w_\lambda(x,s)=
	\begin{cases}
	    u_0(x)^{-\eta}+\lambda f(x,u_0(x))&\text{if }s\leq u_0(x),\\
	    s^{-\eta}+\lambda f(x,s)&\text{if }u_0(x)< s.
	\end{cases}
    \end{align}
    Also, using \eqref{23}, we can consider the truncation $\hat{w}_\lambda\colon\Omega\times\R\to\R$ of $w_\lambda(x,\cdot)$ defined by
    \begin{align}\label{25}
	\hat{w}_\lambda(x,s)=
	\begin{cases}
	    w_\lambda(x,s)&\text{if }s\leq u_\vartheta(x),\\
	    w_\lambda(x,u_\vartheta(x))&\text{if }u_\vartheta(x)< s.
	\end{cases}
    \end{align}
    It is clear that both are Carath\'{e}odory function. We set
    \begin{align*}
	W_\lambda(x,s)=\int^s_0 w_\lambda(x,t)\,dt\quad\text{and}\quad
	\hat{W}_\lambda(x,s)=\int^s_0 \hat{w}_\lambda(x,t)\,dt
    \end{align*}
    and consider the $C^1$-functionals $\sigma_\lambda, \hat{\sigma}_\lambda\colon\Wpzero{p}\to\R$ defined by
    \begin{align*}
	\sigma_\lambda(u)
	&=\frac{1}{p}\|\nabla u\|_p^p +\frac{1}{q}\|\nabla u\|_q^q-\into W_\lambda(x,u)\,dx\quad\text{for all }u\in\Wpzero{p},\\
	\hat{\sigma}_\lambda(u)
	&=\frac{1}{p}\|\nabla u\|_p^p +\frac{1}{q}\|\nabla u\|_q^q-\into \hat{W}_\lambda(x,u)\,dx \quad\text{for all }u\in\Wpzero{p}.
    \end{align*}
    From \eqref{24} and \eqref{25} it is clear that
    \begin{align}\label{26}
	\sigma_\lambda \big|_{[0,u_\vartheta]}=\hat{\sigma}_\lambda \big|_{[0,u_\vartheta]}
	\quad\text{and}\quad
	\sigma'_\lambda \big|_{[0,u_\vartheta]}=\hat{\sigma}'_\lambda \big|_{[0,u_\vartheta]}.
    \end{align}
    Using \eqref{24}, \eqref{25} and the nonlinear regularity theory of Lieberman \cite{Lieberman-1991} we obtain that
    \begin{align}\label{27}
	K_{\sigma_\lambda} \subseteq [u_0)\cap \interior
	\quad\text{and}\quad
	K_{\hat{\sigma}_\lambda}\subseteq [u_0,u_\vartheta]\cap \interior.
    \end{align}
    From \eqref{27} we see that we may assume that
    \begin{align}\label{28}
	K_{\sigma_\lambda} \text{ is finite and }
	K_{\sigma_\lambda} \cap [u_0,u_\vartheta]=\{u_0\}.
    \end{align}
    Otherwise we already have a second positive smooth solution larger that $u_0$ and so we are done.
    
    From \eqref{25} and since $u_0^{-\eta} \in \Lp{1}$, it is clear that $\hat{\sigma}_\lambda$ is coercive and it is also sequentially weakly lower semicontinuous. Hence, we find its global minimizer $\tilde{u}_0 \in \Wpzero{p}$ such that
    \begin{align*}
	\hat{\sigma}_\lambda\l(\tilde{u}_0\r)=\min \l[\hat{\sigma}_\lambda(u)\,:\,u\in\Wpzero{p}\r].
    \end{align*}
    By \eqref{27} we see that $\tilde{u}_0\in K_{\hat{\sigma}_\lambda}\subseteq [u_0,u_\vartheta]\cap\interior$. Then, \eqref{26} and \eqref{28} imply $\tilde{u}_0=u_0\in\interior$. Finally, from \eqref{23} we obtain that $u_0$ is a local $C^1_0(\close)$-minimizer of $\sigma_\lambda$ and then by Gasi\'{n}ski-Papageorgiou \cite{Gasinski-Papageorgiou-2012} we have that 
    \begin{align}\label{29}
	u_0 \text{ is also a local }\Wpzero{p}\text{-minimizer of }\sigma_\lambda.
    \end{align}

    From \eqref{29}, \eqref{28} and Theorem 5.7.6 of Papageorgiou-R\u{a}dulescu-Repov\v{s} \cite[p.\,449]{Papageorgiou-Radulescu-Repovs-2019} we know that we can find $\rho\in(0,1)$ small enough such that
    \begin{align}\label{30}
	\sigma_\lambda(u_0)<\inf \l[\sigma_\lambda(u):\|u-u_0\|=\rho\r]=m_\lambda.
    \end{align}
    
    Hypothesis H(ii) implies that if $u\in\interior$, then
    \begin{align}\label{31}
	\sigma_\lambda(tu)\to -\infty \quad\text{as }t\to +\infty.
    \end{align}
    
    {\bf Claim: } The functional $\sigma_\lambda$ satisfies the C-condition.
    
    Consider a sequence $\{u_n\}_{n\geq 1}\subseteq \Wpzero{p}$ such that
    \begin{align}
	|\sigma_\lambda(u_n)|\leq c_6\quad \text{for some }c_6>0 \text{ and for all }n\in\N,\label{32}\\
	(1+\|u_n\|)\sigma'_\lambda(u_n) \to 0 \text{ in } W^{-1,p'}(\Omega) \text{ as }n\to \infty.\label{33}
    \end{align}

    From \eqref{33} we have
    \begin{align}\label{34}
	\l|\l\lan A_p(u_n),h\r\ran+\l\lan A_q(u_n),h\r\ran - \into w_\lambda (x,u_n)h\,dx \r| \leq \frac{\eps_n \|h\|}{1+\|u_n\|}
    \end{align}
    for all $h\in\Wpzero{p}$ with $\eps_n\to 0^+$. We choose $h=-u_n^-\in\Wpzero{p}$ in \eqref{34} and obtain, by applying \eqref{24}, that
    \begin{align*}
	\l\|u_n^-\r\|^p \leq c_7 \quad\text{for some }c_7>0 \text{ and for all }n\in\N.
    \end{align*}
    This shows that
    \begin{align}\label{35}
	\l\{u_n^-\r\}_{n\geq 1} \subseteq \Wpzero{p} \text{ is bounded}.
    \end{align}
    From \eqref{32} and \eqref{35} it follows that
    \begin{align}\label{36}
	\l\|\nabla u_n^+\r\|_p^p +\frac{p}{q}\l\|\nabla u_n^+\r\|_q^q-\into pF\l(x,u_n^+\r)\,dx \leq c_8\l[1+\l\|u_n^+\r\|_\tau\r]
    \end{align}
    for some $c_8>0$ and for all $n \in \N$, see \eqref{24}.
    Moreover, choosing $h=u_n^+\in\Wpzero{p}$ in \eqref{34}, we obtain by using \eqref{24}
    \begin{align}\label{37}
	-\l\|\nabla u_n^+\r\|_p^p-\l\|\nabla u_n^+\r\|_q^q +\into f\l(x,u_n^+\r)u_n^+\,dx \leq c_9
    \end{align}
    for some $c_9>0$ and for all $n\in\N$. Adding \eqref{36} and \eqref{37} and recall that $q<p$, gives
    \begin{align}\label{38}
	\into \big[ f\l(x,u_n^+\r)u_n^+-pF\l(x,u_n^+\r)\big]\,dx\leq c_{10}\l[1+\l\|u_n^+\r\|_\tau\r]
    \end{align}
    for some $c_{10}>0$ and for all $n\in\N$.
    
    Taking hypotheses H(i), (iii) into account, we see that we can find constants $c_{11}, c_{12}>0$ such that
    \begin{align}\label{39}
	c_{11} s^\tau -c_{12} \leq f(x,s)s-pF(x,s)\quad\text{for a.\,a.\,}x\in\Omega \text{ and for all }s\geq 0.
    \end{align}
    Applying \eqref{39} in \eqref{38}, we infer that
    \begin{align*}
	\l\|u_n^+\r\|_\tau^{\tau-1} \leq c_{13}
    \end{align*}
    for some $c_{13}>0$ and for all $n\in\N$. Therefore,
    \begin{align}\label{40}
	\l\{u_n^+\r\}_{n\geq 1} \subseteq \Lp{\tau} \text{ is bounded}.
    \end{align}

    First assume that $p\neq N$. From hypothesis H(iii), we see that we can always assume that $\tau<r<p^*$. So, we can find $t\in (0,1)$ such that
    \begin{align}\label{41}
	\frac{1}{r}=\frac{1-t}{\tau}+\frac{t}{p^*}.
    \end{align}
    Invoking the interpolation inequality, see Papageorgiou-Winkert \cite[Proposition 2.3.17, p.\,116]{Papageorgiou-Winkert-2018}, we have
    \begin{align*}
	\l\|u_n^+\r\|_r \leq \l\|u_n^+\r\|_\tau^{1-r} \l\|u_n^+\r\|^t_{p^*}.
    \end{align*}
    Hence, by \eqref{40},
    \begin{align}\label{42}
	\l\|u_n^+\r\|_r^r \leq c_{14} \l\|u_n^+\r\|^{tr}
    \end{align}
    for some $c_{14}>0$ and for all $n\in\N$. We choose $h=u_n^+\in \Wpzero{p}$ in \eqref{34} to get
    \begin{align*}
	\l\|u_n^+\r\|^p \leq \into w_\lambda\l(x,u_n^+\r)u_n^+\,dx.
    \end{align*}
    Then, from \eqref{24} and hypothesis H(i), it follows that
    \begin{align*}
	\l\|u_n^+\r\|^p \leq \into c_{15} \l[1+\l(u_n^+\r)^r\r]\,dx
    \end{align*}
    for some $c_{15}>0$ and for all $n\in\N$. This implies
    \begin{align*}
	\l\|u_n^+\r\|^p \leq  
	c_{16}  \l[1+\l\|u_n^+\r\|_r^r\r]
    \end{align*}
    for some $c_{16}>0$ and for all $n\in\N$. Finally, from \eqref{42}, we then obtain
    \begin{align}\label{43}
	\l\|u_n^+\r\|^p \leq  
	c_{17}  \l[1+\l\|u_n^+\r\|^{tr}\r]
    \end{align}
    for some $c_{17}>0$ and for all $n\in\N$.
    
    If $N<p$, then $p^*=\infty$ and so from \eqref{41} we have $tr=r-\tau$, which by hypothesis H(iii) leads to $tr<p$.
    
    If $N>p$, then $p^*=\frac{Np}{N-p}$. From \eqref{41} it follows
    \begin{align*}
	tr=\frac{(r-\tau)p^*}{p^*-\tau},
    \end{align*}
    which implies
    \begin{align*}
	tr=\frac{(r-\tau)Np}{N(p-\tau)+\tau p}<p.
    \end{align*}
    
    Therefore, from \eqref{43} we infer that
    \begin{align}\label{44}
	\l\{u_n^+\r\}_{n\geq 1} \subseteq \Wpzero{p}\text{ is bounded.}
    \end{align}
    
    If $N=p$, then by the Sobolev embedding theorem, we know that $\Wpzero{p}\hookrightarrow \Lp{s}$ continuously for all $1\leq s<\infty$. So, for the argument above to work, we need to replace $p^*$ by $s>r>\tau$ in \eqref{41} which yields
    \begin{align*}
	\frac{1}{r}=\frac{1-t}{\tau}+\frac{t}{s}.
    \end{align*}
    Then, by hypothesis H(iii), we obtain
    \begin{align*}
	tr=\frac{(r-\tau)s}{s-\tau}\to r-\tau <p \quad\text{as }s\to +\infty.
    \end{align*}
    We choose $s>r$ large enough so that $tr<p$. Then, we reach again \eqref{44}.

    From \eqref{44} and \eqref{35} it follows that
    \begin{align*}
	\l\{u_n\r\}_{n\geq 1}\subseteq \Wpzero{p}\text{ is bounded}.
    \end{align*}
    So, we may assume that
    \begin{align}\label{45}
	u_n\weak u \text{ in }\Wpzero{p}\quad\text{and}\quad u_n\to u \text{ in }\Lp{r}.
    \end{align}
    
    In \eqref{34} we choose $h=u_n-u\in\Wpzero{p}$, pass to the limit as $n\to \infty$ and use \eqref{45}. This gives
    \begin{align*}
	\lim_{n\to\infty} \l[\l\lan A_p(u_n),u_n-u\r\ran+\l\lan A_q(u_n),u_n-u\r\ran\r] =0.
    \end{align*}
    The monotonicity of $A_q$ implies
    \begin{align*}
	\lim_{n\to\infty} \l[\l\lan A_p(u_n),u_n-u\r\ran+\l\lan A_q(u),u_n-u\r\ran\r] \leq 0
    \end{align*}
    and from \eqref{45} one has
    \begin{align*}
	\limsup_{n\to\infty} \l\lan A_p(u_n),u_n-u\r\ran\leq 0.
    \end{align*}
    Hence, by Proposition \ref{prop_1}, it follows
    \begin{align*}
	u_n\to u \quad\text{in }\Wpzero{p}.
    \end{align*}
    
    Therefore, $\sigma_\lambda$ satisfies the C-condition and this proves the Claim.
    
    Then, \eqref{30}, \eqref{31} and the Claim permit the use of the mountain pass theorem. So, we can find $\hat{u}\in\Wpzero{p}$ such that
    \begin{align}\label{46}
	\hat{u} \in K_{\sigma_\lambda}\subseteq [u_0)\cap\interior\quad\text{and}\quad
	\sigma_\lambda(u_0)<m_\lambda\leq \sigma_\lambda\l(\hat{u}\r),
    \end{align}
    see \eqref{27} and \eqref{30}, respectively.
    
    From \eqref{46}, \eqref{24} and \eqref{34}, we conclude that
    \begin{align*}
	\hat{u} \in\mathcal{S}_\lambda \subseteq \interior,\quad u_0\leq \hat{u}, \quad u_0 \neq \hat{u}.
    \end{align*}
\end{proof}

\begin{proposition}
    If hypotheses H hold, then $\lambda^*\in\mathcal{L}$.
\end{proposition}

\begin{proof}
    Let $0<\lambda_n <\lambda^*$ with $n\in\N$ and assume that $\lambda_n\nearrow \lambda^*$. By Proposition \ref{prop_5} we can find $u_n\in\mathcal{S}_{\lambda_n}\subseteq \interior$ such that
    \begin{align*}
	\underline{u}\leq u_n\quad\text{for all }n\in\N
    \end{align*}
    and
    \begin{align}\label{48}
	\l\lan A_p(u_n),h\r\ran+\l\lan A_q(u_n),h\r\ran =\into \l[ u_n^{-\eta}+\lambda_n f(x,u_n)\r]h\,dx 
    \end{align}
    for all $h\in\Wpzero{p}$ and for all $n\in\N$. From hypothesis H(iii), we have
    \begin{align}\label{49}
	\ph_\lambda(u_n) \leq c_{18}
    \end{align}
    for some $c_{18}>0$ and for all $n\in\N$, where $\ph_\lambda$ is the energy functional of problem \eqref{problem}.
    
    From \eqref{48}, \eqref{49} and reasoning as in the Claim in the proof of Proposition \ref{prop_11}, we obtain that
    \begin{align}\label{50}
	u_n\to u_*\quad\text{in }\Wpzero{p}.
    \end{align}
    
    So, if in \eqref{48} we pass to the limit as $n\to\infty$ and use \eqref{50}, then
    \begin{align*}
	\l\lan A_p(u_*),h\r\ran+\l\lan A_q(u_*),h\r\ran =\into \l[ u_*^{-\eta}+\lambda^* f(x,u_*)\r]h\,dx 
    \end{align*}
    for all $h\in \Wpzero{p}$ and $\underline{u}\leq u_*$. It follows that $u_*\in\mathcal{S}_{\lambda^*}\subseteq \interior$ and so $\lambda^*\in\mathcal{L}$.
\end{proof}

Therefore, we have
\begin{align*}
    \mathcal{L}=\l(0,\lambda^*\r].
\end{align*}

We can state the following bifurcation-type theorem describing the variations in the set of positive solutions as the parameter $\lambda$ moves in $(0,+\infty)$.

\begin{theorem}
    If hypotheses H hold, then there exist $\lambda^*>0$ such that
    \begin{enumerate}
	\item[(a)]
	    for every $0<\lambda<\lambda^*$, problem \eqref{problem} has at least two positive solutions
	    \begin{align*}
		u_0, \hat{u} \in \interior, \quad u_0 \leq \hat{u},\quad u_0\neq \hat{u};
	    \end{align*}
	  \item[(b)]
	      for $\lambda=\lambda^*$, problem \eqref{problem} has at least one positive solution
	      \begin{align*}
		  u_*\in\interior;
	      \end{align*}
	  \item[(c)]
	      for every $\lambda>\lambda^*$, problem \eqref{problem} has no positive solutions.
    \end{enumerate}
\end{theorem}

\section{Minimal positive solutions} 

In this section we show that for every $\lambda\in\mathcal{L}=(0,\lambda^*]$, problem \eqref{problem} has a smallest positive solutions $u^*\in\interior$ and we investigate the monotonicity and continuity properties of the map $\lambda\to u^*_\lambda$.

\begin{proposition}
    If hypotheses H hold and $\lambda\in\mathcal{L}$, then problem \eqref{problem} has a smallest positive solution $u^*_\lambda \in\mathcal{S}_\lambda\subseteq \interior$, that is, $u^*_\lambda\leq u$ for all $u \in \mathcal{S}_\lambda$.
\end{proposition}

\begin{proof}
    From Proposition 18 of Papageorgiou-R\u{a}dulescu-Repov\v{s} \cite{Papageorgiou-Radulescu-Repovs-2020} we know that the set $\mathcal{S}_\lambda\subseteq \Wpzero{p}$ is downward directed. So, invoking Lemma 3.10 of Hu-Papageorgiou \cite[p.\,178]{Hu-Papageorgiou-1997}, we can find a decreasing sequence $\{u_n\}_{n\geq 1}\subseteq \mathcal{S}_\lambda$ such that
    \begin{align}\label{51}
	\underline{u} \leq u_n\leq u_1 \text{ for all }n\in\N,\quad \inf_{n\geq 1} u_n = \inf \mathcal{S}_\lambda,
    \end{align}
    see Proposition \ref{prop_5}. From \eqref{51} we see that $\{u_n\}_{n\geq 1}\subseteq \Wpzero{p}$ is bounded. From this, as in the proof of Proposition \ref{prop_11}, using Proposition \ref{prop_1}, we obtain
    \begin{align*}
	u_n \to u^*_\lambda \quad\text{in }\Wpzero{p}, \quad \underline{u}\leq u^*_\lambda.
    \end{align*}
    From \eqref{51} it follows
    \begin{align*}
	u^*_\lambda \in \mathcal{S}_\lambda \subseteq \interior\quad\text{and}\quad u^*_\lambda=\inf \mathcal{S}_\lambda.
    \end{align*}
\end{proof}

In the next proposition we examine the monotonicity and continuity properties of the map $\lambda \to u^*_\lambda$ from $\mathcal{L}=(0,\lambda^*]$ into $C^1_0(\close)$.

\begin{proposition}
    If hypotheses H hold, then the minimal solution map $\lambda \to u^*_\lambda$ from $\mathcal{L}=(0,\lambda^*]$ into $C^1_0(\close)$ is
    \begin{enumerate}
	\item[(a)]
	    strictly increasing in the sense that
	    \begin{align*}
		0<\mu<\lambda\leq \lambda^* \quad\text{implies}\quad
		u^*_\lambda-u^*_\mu \in \interior;
	    \end{align*}
	\item[(b)]
	    left continuous.
    \end{enumerate}
\end{proposition}

\begin{proof}
    (a) Let $0<\mu<\lambda\leq \lambda^*$. According to Proposition \ref{prop_5} we can find $u_\mu\in\mathcal{S}_\mu\subseteq \interior$ such that $u^*_\lambda-u_\mu\in\interior$. Since $u^*_\lambda\leq u_\mu$ we obtain the desired conclusion.
    
    (b) Suppose that $\lambda_n \to \lambda^- \leq \lambda^*$. Then $\{u_n^*\}_{n\geq 1}:=\{u^*_{\lambda_n}\}_{n\geq 1}\subseteq \interior$ is increasing and
    \begin{align}\label{52}
	\underline{u} \leq u_n^* \leq u^*_{\lambda^*}\quad\text{for all }n\in\N.
    \end{align}
    From \eqref{52} and the nonlinear regularity theory of Lieberman \cite{Lieberman-1991} we have that $\{u^*_n\}_{n\geq 1}\subseteq C^1_0(\close)$ is relatively compact and so
    \begin{align}\label{53}
	u_n^*\to \tilde{u}^*_\lambda \quad\text{in }C^1_0(\close).
    \end{align}
    
    If $\tilde{u}^*_\lambda\neq u^*_\lambda$, then we can find $z_0\in\Omega$ such that
    \begin{align*}
	u^*_\lambda(z_0)<\tilde{u}^*_\lambda(z_0).
    \end{align*}
    From \eqref{53} we then derive
    \begin{align*}
	u^*_\lambda(z_0)<u_n^*(z_0)\quad\text{for all }n \geq n_0,
    \end{align*}
    which contradicts (a). So, $\tilde{u}^*_\lambda=u^*_\lambda$ and we conclude the left continuity of $\lambda\to u^*_\lambda$.
\end{proof}

Summarizing our findings in this section, we can state the following theorem.

\begin{theorem}
    If hypotheses H hold and $\lambda\in\mathcal{L}=(0,\lambda^*]$, then problem \eqref{problem} admits a smallest positive solution $u^*_\lambda\in\mathcal{S}_\lambda \subseteq \interior$ and the map $\lambda\to u^*_\lambda$ from $\mathcal{L}=(0,\lambda^*]$ into $C^1_0(\close)$ is
    \begin{enumerate}
	\item[(a)]
	    strictly increasing;
	\item[(b)]
	    left continuous.
    \end{enumerate}
\end{theorem}


\end{document}